\documentclass[10pt,a4paper]{amsart}
\usepackage[english]{babel}

\usepackage{amsmath,amssymb,amsthm}

\theoremstyle{plain}
\newtheorem{theorem}{Theorem}[section]
\newtheorem{corollary}[theorem]{Corollary}
\newtheorem{prop}[theorem]{Proposition}
\newtheorem{lemma}[theorem]{Lemma}
\theoremstyle{definition}
\newtheorem{remark}[theorem]{Remark}

\newtheorem{example}[theorem]{Example}

\newcommand{\C}{\mathbb{C}}

\newcommand{\R}{\mathbb{R}}
\newcommand{\N}{\mathbb{N}}

\newcommand{\Pol}[2][k]{\ensuremath{\mathcal{P}\left(^{#1}#2;#2\right)}}

\newcommand{\eps}{\varepsilon}

\DeclareMathOperator{\e}{e}

\newcounter{equi1}

\renewcommand{\leq}{\leqslant}
\renewcommand{\geq}{\geqslant}

\begin{document}
\begin{center}
\texttt{[Linear Algebra Appl.\ (2009),
doi:10.1016/j.laa.2008.12.020]}
\end{center}

\title[Spaces with Polynomial numerical index zero]
{Two-dimensional Banach spaces with Polynomial numerical index zero}

\author[D.~Garc\'{\i}a]{Domingo Garc\'{\i}a}
\author[B.~C.~Grecu]{Bogdan C.\ Grecu}
\address[Grecu]{Department of Pure Mathematics, Queen's University Belfast,
BT7 1NN, United Kingdom} \email{b.grecu@qub.ac.uk}

\author[M.~Maestre]{Manuel Maestre}
\address[Garc\'{\i}a \& Maestre]{Departamento de An\'{a}lisis Matem\'{a}tico,
Universidad de Valencia, Doctor Moliner $50$, $46100$ Burjasot
(Valencia), Spain} \email{domingo.garcia@uv.es,
manuel.maestre@uv.es}

\author[M.~Mart\'{\i}n]{Miguel Mart\'{\i}n}
\author[J.~Mer\'{\i}]{Javier Mer\'{\i}}

\address[Mart\'{\i}n \& Mer\'{\i}]{Departamento de An\'{a}lisis Matem\'{a}tico, Facultad de Ciencias,
Universidad de Granada, 18071 Granada, Spain}
\email{mmartins@ugr.es, jmeri@ugr.es}



\thanks{The first and third authors were
supported by Spanish MICINN Project MTM2008-03211/MTM. This article was started while the second
author was a postdoctoral fellow at Departamento de An\'{a}lisis Matem\'{a}tico, Universidad de Valencia.
He was supported  by a Marie Curie Intra European Fellowship (MEIF-CT-2005-006958). The fourth and
fifth authors were supported by Spanish MEC Project MTM2006-04837 and Junta de Andaluc\'{\i}a grants
FQM-185 and P06-FQM-01438.}

\begin{abstract}
We study two-dimensional Banach spaces with polynomial
numerical indices equal to zero.
\end{abstract}

\maketitle

\section{Introduction}
The polynomial numerical indices of a Banach space are constants
relating the norm and the numerical radius of homogeneous
polynomials on the space. Let us present the relevant definitions.
For a Banach space $X$, we write $B_X$ for the closed unit ball,
$S_X$ for the unit sphere, $X^*$ for the dual space, and $\Pi(X)$
for the subset of $X\times X^*$ given by
$$
\Pi(X)=\left\{(x,x^*)\in S_X\times S_{X^*}\ : \ x^*(x)=1\right\}.
$$
For $k\in\N$ we denote by $\Pol[k]{X}$ the space of all
$k$-homogeneous polynomials from $X$ into $X$ endowed with the norm
$$
\|P\|=\sup\{\|P(x)\|\ : \ x\in B_X\}.
$$
We recall that a mapping $P:X\longrightarrow X$ is called a
(continuous) $k$-homogeneous polynomial on $X$ if there is a
$k$-linear continuous mapping $A:X\times \cdots \times X
\longrightarrow X$ such that $P(x)=A(x,\ldots,x)$ for every
$x\in X$. We refer to the book \cite{Dineenbook} for
background. Given $P\in\Pol[k]{X}$, the \emph{numerical range}
of $P$ is the subset of the scalar field given by
$$
V(P)=\{x^*(P (x))\ : \ (x,x^*)\in \Pi(X)\},
$$
and the \emph{numerical radius} of $P$ is
$$
v(P)=\sup \{|x^*(P (x))|\ : \ (x,x^*)\in \Pi(X)\}.
$$

Recently, Y.~S.~Choi, D.~Garc\'{\i}a, S.~G.~Kim and M.~Maestre
\cite{C-G-K-M} have introduced the \emph{polynomial numerical
index of order $k$} of a Banach space $X$ as the constant
$n^{(k)}(X)$ defined by
\begin{align*}
n^{(k)}(X)&=\max\left\{c\geq 0\ : \ c\,\|P\| \leqslant v(P) \ \
\forall\, P\in \Pol[k]{X}\right\}\\
&=\inf\left\{v(P)\ : \ P\in \Pol[k]{X},\ \|P\|=1\right\}
\end{align*}
for every $k\in\N$. This concept is a generalization of the
\emph{numerical index} of a Banach space (recovered for $k=1$)
which was first suggested by G.~Lumer in 1968 \cite{D-Mc-P-W}.

Let us recall some facts about the polynomial numerical index
which are relevant to our discussion. We refer the reader to
the already cited \cite{C-G-K-M} and to
\cite{ChGarMaMa-QJM,KMM,Lee} for recent results and background.
The easiest examples are $n^{(k)}(\R)=1$ and $n^{(k)}(\C)=1$
for every $k\in \N$. In the complex case, $n^{(k)}(C(K))=1$ for
every $k\in \N$ and $n^{(2)}(\ell_1)\leq \frac12$. The real
spaces $\ell_1^{m}$, $\ell_\infty^{m}$, $c_0$, $\ell_1$ and
$\ell_\infty$ have polynomial numerical index of order $2$
equal to $1/2$ (\cite{KMM}). The only finite-dimensional real
Banach space $X$ with $n^{(2)}(X)=1$ is $X=\R$ (\cite{Lee}).
The inequality $ n^{(k+1)}(X)\leq n^{(k)}(X) $ holds for every
real or complex Banach space $X$ and every $k\in \N$, giving
that $n^{(k)}(H)=0$ for every $k\in \N$ and every real Hilbert
space $H$ of dimension greater than one. This last fact is not
true in the complex case in which it follows from an old result
by Harris (\cite{Harris}) that $n^{(k)}(X)\geq
k^{\frac{k}{1-k}} $ for every complex Banach space $X$ and
every $k\geq 2$. Finally, $n^{(k)}(X^{**})\leq n^{(k)}(X)$ for
every real or complex Banach space $X$ and every $k\in \N$, and
this inequality may be strict.

For a real finite-dimensional space $X$, the fact $n(X)=0$ is
equivalent to $X$ having infinitely many surjective isometries
\cite[Theorem~3.8]{Ros1984}. In particular, it can be shown
that the only two-dimensional space with infinitely many
surjective isometries is the Hilbert space. For bigger
dimensions the situation is not that easy but it is possible to
somehow describe all these spaces (see \cite{M-M-R} and
\cite{Ros1984}).

We will show in this paper that the situation for numerical
indices of higher order is not so tidy, and many different
examples of two-dimensional spaces with numerical indices of
higher order equal to zero will be given. Namely, we start by
showing that $n^{(p-1)}(\ell_p^2)=0$ if $p$ is an even number
and, actually, that $n^{(2k-1)}(X)=0$ if $(X, \|\cdot\|)$ is a
real Banach space of dimension greater than one such that the
mapping $x \longmapsto \|x\|^{2k}$ is a $2k$-homogeneous
polynomial. Next, we describe all absolute normalized and
symmetric norms on $\R^2$ such that the polynomial numerical
index of order $3$ is zero showing, in particular, that all
these norms come from a polynomial. Finally, we present some
examples proving that the situation is different for higher
orders and for nonsymmetric norms. This is the content of
section~\ref{sec:2}. We include an appendix
(section~\ref{sec:appendix}) where it is shown that the
formulae appearing in the examples are actually norms on
$\R^2$.

Let us finish the introduction with some notation. We say that
a norm $\|\cdot\|$ in $\R^2$ is \emph{absolute} if
$\|(x,y)\|=\|(|x|,|y|)\|$ for every $x,y\in \R$,
\emph{normalized} if $\|(1,0)\|=\|(0,1)\|=1$ and
\emph{symmetric} whenever $\|(x,y)\|=\|(y,x)\|$ for every
$x,y\in \R$. For $1\leq p \leq \infty$, we write $\|\cdot\|_p$
to denote the $p$-norm and $\ell_p^{d}$ to denote the
$d$-dimensional $\ell_p$-space (i.e.\ $\R^d$ endowed with
$\|\cdot\|_p$).

Let $X$ be a Banach space, $k\in \N$ and let $S\in L(X)$ be a
surjective isometry. Given $P\in \Pol[k]{X}$, $S^{-1}\circ P
\circ S \in \Pol[k]{X}$ clearly and one has that
\begin{equation}\label{eq-hom-num-range}
V(S^{-1}\circ P \circ S)=V(P) \quad \text{and} \quad \|S^{-1}\circ P \circ S\|=\|P\|
\end{equation}
(indeed, these equalities follow easily from
\cite[Theorem~2]{Harris} but they are actually straightforwardly
deduced from the definition of numerical range).

Let us also recall that $X$ is a \emph{smooth} space if given
$x\in X\setminus \{0\}$ there exists a unique norm-one linear
functional $x^*\in X^*$ such that $x^*(x)=\|x\|$. Moreover,
this functional is given by the derivative $D_x\|\cdot\|$ of
the norm at $x$. If $X$ is a finite-dimensional space it is
known \cite[Corollary 1.5 and Remark 1.7]{DGZ} that $X$ is
smooth if and only if its norm is Fr\'{e}chet differentiable on
$S_X$.

\section{The results}\label{sec:2}
Our first goal is to discuss the polynomial numerical index of
the real spaces $\ell_p^2$ for $1<p<\infty$. Let us recall that
$n^{(k)}(\ell_p^2)>0$ for $p=1,\infty$ and every $k\in\N$
\cite[Corollary~2.5]{KMM}.

\begin{example}\label{example:l-p} Let $1<p<\infty$.
\begin{enumerate}
\item[(a)] If $p$ is an even number and $k\in \N$, then
    $n^{(k)}(\ell_p^2)=0$ if $k\geq p-1$ and
    $n^{(k)}(\ell_p^2)>0$ if $k<p-1$.
\item[(b)] If $p$ is not an even number, then
    $n^{(k)}(\ell_p^2)>0$ for every $k\in \N$.
\end{enumerate}
\end{example}

\begin{proof}
(a). Given $(x,y)\in S_{\ell_p^2}$\,, the only functional which
norms $(x,y)$ is $(x^{p-1}, y^{p-1}) \in \ell_{p/p-1}^2$. If we
consider the polynomial $P\in\Pol[p-1]{\ell_p^2}$ defined by
$P(x, y)=(-y^{p-1}, x^{p-1})$ then,
$$
(x^{p-1},
y^{p-1})(P(x,y))=-x^{p-1} y^{p-1} + y^{p-1} x^{p-1} = 0
$$
for all $(x,y) \in S_{\ell_p^2}$ implying that $v(P)=0$ and
$n^{(p-1)}(\ell_p^2)=0$. Therefore, for $k\geq p-1$,
$n^{(k)}(\ell_p^2)=0$ by \cite[Proposition 2.5]{C-G-K-M}. If $k
< p-1$ and $P=(P_1,P_2)\in \Pol[k]{\ell_p^2}$ is non zero,
observe that
$$
x^{p-1}P_1(x,y)+ y^{p-1}P_2(x,y)
$$
is a scalar homogeneous polynomial which cannot be constant
zero. Indeed, we can assume without loss of generality that
$P_1$ is non-zero and evaluate the above expression at $(x,1)$
for $x\in \R$ obtaining
$$
x^{p-1}P_1(x,1)+ P_2(x,1).
$$
We observe that the first summand is a non-zero polynomial in
the variable $x$ of degree at least $p-1$ and the second one
has degree at most $k$. So their sum cannot be equal to zero
for every $x\in\R$.

(b). When $p$ is not an even number, the only linear functional
which norms $(x,y)\in \ell_p^2$ with $x,y\neq 0$ is
$(x|x|^{p-2}, y|y|^{p-2}) \in \ell_{p/p-1}^2$. If
$P=(P_1,P_2)\in \Pol[k]{\ell_p^2}$ satisfies $v(P)=0$, then
\begin{equation}\label{eq:duality-lp}
x|x|^{p-2} P_1(x,y) + y |y|^{p-2} P_2(x,y)=0
\end{equation}
for every $x,y\neq 0$. Now, if $p\notin \N$, evaluating at $(x,1)$
for every $x>0$, we get
$$
x^{p-1} P_1(x,1)=-P_2(x,1) \qquad (x\in \R^+).
$$
If $P_1(x,1)$ is not zero in $\R^+$, dividing the above
equation by $x^{p-1+\text{deg}(P_1(x,1))}$ and taking the limit
as $x\rightarrow +\infty$, we get a contradiction. Hence, we
have that $P_1(x,1)=0$ for $x\in \R^+$ which implies
$P_2(x,1)=0$ for $x\in \R^+$ and, therefore, that $P=0$.
Finally, if $p\in\N$ is odd, we use \eqref{eq:duality-lp} to
obtain
\begin{align*}
x^{p-1} P_1(x,1)+P_2(x,1)&=0 \qquad (x\in \R^+)\\
-x^{p-1}P_1(x,1)+P_2(x,1)&=0 \qquad (x\in \R^-)
\end{align*}
which, together with the fact that $x^{p-1} P_1(x,1)+P_2(x,1)$ and
$-x^{p-1} P_1(x,1)+P_2(x,1)$ are polynomials, implies
\begin{align*}
x^{p-1} P_1(x,1)+P_2(x,1)&=0 \qquad (x\in \R)\\
-x^{p-1}P_1(x,1)+P_2(x,1)&=0 \qquad (x\in \R).
\end{align*}
This obviously gives $P_1(x,1)=0$ and $P_2(x,1)=0$ for $x\in\R$,
implying that $P=0$ and finishing the proof.
\end{proof}

Since $\ell_p^2$ is an absolute summand of $\ell_p$ and
$\ell_p^d$ for every $d\geq2$, by
\cite[Proposition~2.1]{ChGarMaMa-QJM} we get the following.

\begin{corollary}\label{cor:lp-dim-d-infty}
Let $p$ be an even number and $d\geq 2$ an integer. Then,
$n^{(p-1)}(\ell_p)=n^{(p-1)}(\ell_p^d)=0$.
\end{corollary}

\begin{remark}
It is claimed in \cite{kim} that $n^{(k)}(\ell_p^d)>0$ for
every $k\in\N$, every $1<p<\infty$, $p \neq 2$, and every
integer $d\geq 2$. Going into the proof of that result, one
realizes that it is needed that $p$ is not an even integer.
\end{remark}

It is known that $n(X^\ast) \leq n(X)$ for every Banach space
$X$. Example~\ref{example:l-p} shows that, unlike the linear
case, there is no general inequality between the polynomial
numerical indices of a Banach space and the ones of its dual.

\begin{example}
{\slshape The reflexive space $X=\ell_4^2$ satisfies
$n^{(k)}(X)=0$ and $n^{(k)}(X^*)>0$ for all $k\geq 3$.}
\end{example}

Our next result is a generalization of
Corollary~\ref{cor:lp-dim-d-infty} to every Banach space whose
norm raised to an even power is a homogeneous polynomial.

\begin{prop}\label{prop:norm-polynomial}
Let $k$ be a positive integer and let $(X, \|\cdot\|)$ be a
real Banach space of dimension greater than one. If the mapping
$x \longmapsto \|x\|^{2k}$ is a $2k$-homogeneous polynomial,
then $n^{(2k-1)}(X)=0$.
\end{prop}

\begin{proof}
Let $R$ and $A$ be respectively the $2k$-homogeneous scalar
polynomial and the corresponding symmetric $2k$-linear form
such that $A(x, \ldots, x)=R(x)=\|x\|^{2k}$ for every $x\in X$.
Since $R$ is G\^{a}teaux differentiable on $S_X$ so is $\|\cdot\|$.
Moreover, for fixed $x\in S_X$, we have that
$$
2kD_x\|\cdot\|(y)=D_xR(y)=2kA(x,\ldots,x,y)
$$
for every $y\in X$ and, therefore, the functional given by
$x^*(y)=A(x, \ldots,x, y)$ is the only norm-one functional
satisfying $x^*(x)=1$. To finish the proof, we fix $x_0,y_0$
two linearly independent elements of $X$ and we define
$P\in\Pol[2k-1]{X}$ by
$$
P(x)=-A(x,\ldots, x, y_0)x_0 + A(x,\ldots, x, x_0)y_0 \qquad
\big(x\in X\big)
$$
which clearly satisfies $P \neq 0$. Finally, for $(x,x^*)\in
\Pi(X)$ we have that
\begin{align*}
x^*\big(P(x)\big)&=A\big(x,\ldots,x,P(x)\big)\\
&= A\big(x,\ldots,x,-A(x,\ldots, x, y_0)x_0 + A(x,\ldots, x, x_0)y_0\big)\\
&=-A(x,\ldots, x, y_0)A(x,\ldots, x, x_0) + A(x,\ldots, x,
x_0)A(x,\ldots, x, y_0)=0
\end{align*}
so $v(P)=0$ and, consequently, $n^{(2k-1)}(X)=0$.
\end{proof}

The rest of the paper is devoted to the two-dimensional case.
We start with some facts about two-dimensional spaces with
polynomial numerical index $0$ which will be useful in this
paper.

\begin{theorem}\label{thm:initial-facts}
Let $(X, \|\cdot\|)$ be a two dimensional real space such that
$n^{(k)}(X)=0$ for some $k\geq 1$, let $k_0=\min\{k \, : \,
n^{(k)}(X)=0\}$, and $P=(P_1,P_2)\in \Pol[k_0]{X}$ with
$v(P)=0$. The following hold:
\begin{enumerate}
\item[(a)] The $(k_0+1)$-homogeneous scalar polynomial
    defined by
    $$
    Q(x,y)=yP_1(x,y)-xP_2(x,y)\qquad \big((x,y)\in X\big)
    $$
    only vanishes at $(0,0)$.
\item[(b)] $k_0$ is odd.
\item[(c)] $(X,\|\cdot\|)$ is a smooth space. Moreover, for
    every non-zero $(x,y)\in X$ the unique functional
    $(x^*,y^*)\in S_{X^*}$ which norms $(x,y)$ is given by
$$
x^*= \frac{-P_2(x,y)\|(x,y)\|}{Q(x,y)} \qquad
\text{and} \qquad y^*=
\frac{P_1(x,y)\|(x,y)\|}{Q(x,y)}\,.
$$
\item[(d)] The polynomial $P$ is unique in the following
    sense:  $\widetilde{P}\in \Pol[k_0]{X}$ satisfies
    $v(\widetilde{P})=0$ if and only if there exists
    $\lambda\in\R$ so that $\widetilde{P}=\lambda P$.
\end{enumerate}
\end{theorem}

\begin{proof}
Given $P=(P_1,P_2)\in \Pol[k_0]{X}$ with $v(P)=0$, we claim
that $P_1$ and $P_2$ do not have any factor in common and, in
particular, that $P$ only vanishes at $(0,0)$. Indeed, if
$k_0\geq 2$, suppose that there exist scalar polynomials
$S,R_1,R_2$ with $\text{deg}(R_i)<k_0$ such that $P_i=SR_i$ for
$i=1,2$. Since $v(P)=0$, given an element $(x,y)\in S_X$ and a
linear functional $(x^*,y^*)\in S_{X^*}$ satisfying
$x^*x+y^*y=1$, we have that
$$
x^*P_1(x,y) + y^*P_2(x,y) = 0
$$
and, therefore,
$$
S(x,y)\Big(x^*R_1(x,y) + y^*R_2(x,y)\Big) = 0
$$
which gives us $x^*R_1(x,y) + y^*R_2(x,y) = 0$ whenever
$S(x,y)\neq0$. Writing $R=(R_1,R_2)$ and using that $V(R)$ is
connected \cite[Theorem~1]{B-C-S} and that $S$ only has a finite
number of zeros in $S_X$, we deduce $v(R)=0$ and so $n^{(k)}(X)=0$
for some $k<k_0$, contradicting the minimality of $k_0$. If $k_0=1$,
the above argument is immediate.

(a). The fact that $Q(x_0,y_0)=0$ for some $(x_0,y_0)\neq0$
yields that $P(x_0,y_0) = \lambda (x_0,y_0)$ for some
$\lambda\in \R$ which, together with $v(P)=0$, implies that
$\lambda=0$ contradicting the fact that $P$ only vanishes at
$(0,0)$.

(b). Since $Q$ only vanishes at $(0,0)$, its degree $k_0+1$
must be even and thus $k_0$ is odd.

(c). Given $(x,y)\in S_X$, we observe that any functional
$(x^*,y^*)\in S_{X^*}$ norming $(x,y)$ satisfies the linear
equations $ x^*x + y^*y = 1$ and $x^*P_1(x,y) + y^*P_2(x,y) =
0$ which uniquely determine $(x^*,y^*)$ as
\begin{equation*}
x^*= \frac{-P_2(x,y)}{Q(x,y)} \qquad \text{and} \qquad y^*=
\frac{P_1(x,y)}{Q(x,y)}
\end{equation*}
since $Q(x,y)\neq 0$. For arbitrary $(x,y)\neq (0,0)$ it
suffices to use what we have just proved and the homogeneity.

(d). Since $v(\widetilde{P})=0$, for every
$((x,y),(x^*,y^*))\in\Pi(X)$ we have
$x^*\widetilde{P}_1(x,y)+y^*\widetilde{P}_2(x,y)=0$ which,
together with (c), gives
$$
\frac{-P_2(x,y)}{Q(x,y)}\widetilde{P}_1(x,y)+\frac{P_1(x,y)}{Q(x,y)}
\widetilde{P}_2(x,y)=0
$$
and, therefore,
$$
P_1(x,y)\widetilde{P}_2(x,y)=P_2(x,y)\widetilde{P}_1(x,y)
$$
for every $(x,y)\in S_X$. Now it suffices to recall that $P_1$ and
$P_2$ do not have any factor in common to get the result.
\end{proof}

We have to restrict ourselves to the two-dimensional case since
the above result is not true for higher dimensions.

\begin{remark}\label{rem-dimension-superior}
{\slshape Consider the real Banach space $X=\ell_2^2 \oplus_1
Y$, where $Y$ is any non-null Banach space. Then $n^{(k)}(X)
\leq n^{(k)}(\ell_2^2)=0$ for every $k\in \N$ by
\cite[Proposition~2.1]{ChGarMaMa-QJM}. But the norm of $X$ is
not smooth at points $(0,y)\in S_X$ with $y\in S_Y$. Also, if
we choose $Y$ such that $n^{(k)}(Y)=0$, there are different
non-null polynomials whose numerical radii are zero.}
\end{remark}

A consequence of Theorem~\ref{thm:initial-facts} is the
following partial answer to Problem~42 of \cite{KaMaPa}.

\begin{corollary}
If $X$ is a two-dimensional real Banach space with
$n^{(2)}(X)=0$, then $n(X)=0$.
\end{corollary}

It is a well known result (see \cite[Corollary~2.5]{M-M-R} and
\cite[Theorem~3.1]{Ros1984}) that the only two dimensional real
space with numerical index $0$ is the Euclidean space. The
above theorem allows us to give a different and elementary
proof of this fact. We include it here since it gives some
ideas which we will use later.

\begin{corollary}\label{cor:Hilbert-space}
Let $X$ be a two dimensional real space with $n(X)=0$. Then,
$X$ is the two dimensional real Euclidean space.
\end{corollary}

\begin{proof}
Let $e_1, e_2\in S_X$ and $e_1^*, e_2^*\in S_{X^*}$ be so that
$e_i^*(e_j)=\delta_{ij}$ for $i,j\in\{1,2\}$ (the existence of such
elements is guaranteed by \cite[Theorem~II.2.2]{Singer}). We fix a
linear operator $T$ with $v(T)=0$ and we write it in the basis
$\{e_1, e_2\}$:
$$
T(x,y)=(ax+by, cx+dy) \qquad \big((x,y)\in X\big).
$$
Since $e_i^*(Te_i)=0$ for $i=1,2$ we obtain $a=d=0$. Given an
arbitrary nonzero $(x,y)\in X$, we use
Theorem~\ref{thm:initial-facts} to get that the unique linear
functional which norms $(x,y)$ is given by
$$
\left(\frac{-cx\|(x,y)\|}{by^2-cx^2}\,,
\frac{by\|(x,y)\|}{by^2-cx^2}\right)
$$
but such a functional must coincide with the differential of
the norm, implying that
$$
\frac{\partial\|\cdot\|}{\partial x}(x,y)=
\frac{-cx\|(x,y)\|}{by^2-cx^2} \qquad \text{and} \qquad
\frac{\partial\|\cdot\|}{\partial y}(x,y)=
\frac{by\|(x,y)\|}{by^2-cx^2}.
$$
We rewrite the first equation as follows:
$$
\frac{1}{\|(x,y)\|}\frac{\partial\|\cdot\|}{\partial
x}(x,y)=\frac{-cx}{by^2-cx^2}
$$
and we integrate it with respect to $x$, obtaining
$$
\log \|(x,y)\|= \frac{1}{2}\log(by^2-cx^2) + f(y)
$$
for some differentiable function $f$. Differentiating now with
respect to $y$ we get
$$
\frac{1}{\|(x,y)\|}\frac{\partial\|\cdot\|}{\partial
y}(x,y)=\frac{by}{by^2-cx^2} + f^{\prime}(y)
$$
so $f^{\prime}(y)=0$ and $f(y)$ is constant, say $M$. Therefore, we
can write
$$
\|(x,y)\|= \e^M (by^2-cx^2)^{\frac{1}{2}}
$$
and deduce that $b>0$ and $c<0$. Now, since
$\|e_1\|=\|e_2\|=1$, we get $1 = \e^Mb^{\frac{1}{2}} =
\e^M(-c)^{\frac{1}{2}}$ which yields that
\begin{equation*}
\|(x,y)\|= \e^M (by^2-cx^2)^{\frac{1}{2}}=\e^M
b^{\frac{1}{2}}(x^2+y^2)^{\frac{1}{2}}=(x^2+y^2)^{\frac{1}{2}} \ .\qedhere
\end{equation*}
\end{proof}

There are more two-dimensional spaces for which the polynomial
numerical index of order $3$ is zero since we already know that
$n^{(3)}(\ell_4^2)=0$. However, we are able to completely
describe absolute normalized and symmetric norms with
polynomial numerical index of order $3$ equal to zero showing,
in particular, that all of them come from a polynomial. We will
see later that the hypothesis of symmetry is necessary.

\begin{theorem}\label{th-norm-beta}
Let $X=(\R^2,\|\cdot\|)$ be a two-dimensional Banach space
satisfying that $n^{(3)}(X)=0$ with $\|\cdot\|$ being a
normalized absolute symmetric norm. Then, there is $\beta\in
[0,3]$ so that
$$
\|(x,y)\|=\left(x^4+2\beta x^2y^2+y^4\right)^\frac{1}{4} \qquad \big((x,y)\in
X\big).
$$
In particular, the fourth power of the norm of $X$ is a
polynomial.
\end{theorem}

\begin{proof}
We can assume that $n^{(2)}(X)\neq0$ since otherwise $X$ is a
Hilbert space and the result holds with $\beta=1$. We fix
$P=(P_1,P_2)\in \Pol[3]{X}$ with $v(P)=0$ and we consider the
associated scalar polynomial $Q(x,y)=yP_1(x,y)-xP_2(x,y)$ which
only vanishes at $(0,0)$ by Theorem~\ref{thm:initial-facts}.
Hence we can assume without loss of generality that $Q>0$ on
$\R^2\setminus \{(0,0) \}$. Next, the norm being absolute, the
operator $U=\begin{pmatrix} 1 & 0
\\ 0 & -1
\end{pmatrix}$ is a surjective isometry and so the polynomial
$(R_1,R_2)=U^{-1}\circ P \circ U$, which is given by
$$
\bigl(R_1(x,y),R_2(x,y)\bigr)= \bigl(P_1(x,-y),-P_2(x,-y)\bigr) \qquad
\bigl((x,y)\in X\bigr),
$$
satisfies
$$
v(R_1,R_2)=0 \qquad \text{and} \qquad \|(R_1,R_2)\|=\|P\|
$$
by \eqref{eq-hom-num-range}. Thus,
Theorem~\ref{thm:initial-facts} tells us that there is
$\lambda\in\R$ with $|\lambda|=1$ so that
$$
P_1(x,-y)=\lambda P_1(x,y) \qquad \text{and} \qquad
P_2(x,-y)=-\lambda P_2(x,y)
$$
for every $(x,y)\in X$. Moreover, we have that $\lambda=-1$. Indeed,
it suffices to take a non-zero $x\in\R$ and to observe that
$$
Q(x,-x)=-xP_1(x,-x)-xP_2(x,-x)=-\lambda Q(x,x)
$$
which implies $\lambda=-1$ since $Q>0$ on $\R^2\setminus
\{(0,0) \}$. Hence, for every $(x,y)\in X$ we get
\begin{equation}\label{eq:absoluteness}
P_1(x,-y)=-P_1(x,y) \qquad \text{and} \qquad P_2(x,-y)=P_2(x,y).
\end{equation}
Analogously, the norm being symmetric, the operator
$V=\begin{pmatrix} 0 & 1 \\ 1 & 0
\end{pmatrix}$ is a surjective isometry and so the polynomial
$(S_1,S_2)=V^{-1}\circ P \circ V$, which is given by
$$
\bigl(S_1(x,y),S_2(x,y)\bigr)= \bigl(P_2(y,x),P_1(y,x)\bigr) \qquad
\bigl((x,y)\in X\bigr),
$$
satisfies
$$
v(S_1,S_2)=0 \qquad \text{and} \qquad \|(S_1,S_2)\|=\|P\|.
$$
by \eqref{eq-hom-num-range}. Therefore, using again
Theorem~\ref{thm:initial-facts} and the fact that $Q>0$ on
$\R^2\setminus \{(0,0) \}$, we deduce that
\begin{equation*}
P_2(x,y)=-P_1(y,x) \qquad \big((x,y)\in X\big).
\end{equation*}
Therefore, if we write $P_1(x,y)=a x^3+bx^2y+cxy^2+dy^3$ for some
$a,b,c,d\in\R$, we obtain $P_2(x,y)=-dx^3-cx^2y-bxy^2-ay^3$.
Further, using \eqref{eq:absoluteness} we deduce that
$$
P_1(x,y)=bx^2y+dy^3 \qquad \text{and} \qquad P_2(x,y)=-dx^3-bxy^2
$$
for every $(x,y)\in X$. This, together with
Theorem~\ref{thm:initial-facts}, tells us that the linear
functional which norms an arbitrary non-zero $(x,y)\in X$ is
given by
$$
\left(\frac{(dx^3+bxy^2)\|(x,y)\|}{dx^4+2bx^2y^2+dy^4}\,,
\frac{(bx^2y+dy^3)\|(x,y)\|}{dx^4+2bx^2y^2+dy^4}\right)
$$
thus, we have that
$$
\frac{1}{\|(x,y)\|}\frac{\partial\|\cdot\|}{\partial x}(x,y)=
\frac{dx^3+bxy^2}{dx^4+2bx^2y^2+dy^4} \quad \text{and} \quad
\frac{1}{\|(x,y)\|}\frac{\partial\|\cdot\|}{\partial y}(x,y)=
\frac{bx^2y+dy^3}{dx^4+2bx^2y^2+dy^4}\,.
$$
Integrating the first equation with respect to $x$ we obtain
$$
\log \|(x,y)\|= \frac{1}{4}\log(dx^4+2bx^2y^2+dy^4) + f(y) \qquad (x,y\in \R)
$$
for some differentiable function $f$. Differentiating now with
respect to $y$ we get
$$
\frac{1}{\|(x,y)\|}\frac{\partial\|\cdot\|}{\partial
y}(x,y)=\frac{bx^2y+dy^3}{dx^4+2bx^2y^2+dy^4} + f^{\prime}(y) \qquad (x,y\in \R)
$$
so $f^{\prime}(y)=0$ and $f(y)$ is constant, say $C$. Therefore, we
can write
$$
\|(x,y)\|= \e^C (dx^4+2bx^2y^2+dy^4)^{\frac{1}{4}} \qquad (x,y\in \R).
$$
Now, since $\|(1,0)\|=\|(0,1)\|=1$, $d>0$ and
$\e^Cd^{\frac14}=1$ so, calling $\beta=b\e^{4C}$, we have
$$
\|(x,y)\|= (x^4+2\beta x^2y^2+y^4)^{\frac{1}{4}} \qquad (x,y\in \R).
$$
Finally, this formula defines a norm if and only if $\beta\in
[0,3]$ as shown in Proposition~\ref{prop:Apendix-beta}.
\end{proof}

The next example shows that the hypothesis of symmetry of the norm
in the above theorem cannot be dropped.

\begin{example}\label{ex:raro-dim3}
{\slshape There are normalized absolute norms $\|\cdot\|$ on
$\R^2$ such that the spaces $X=(\R^2,\|\cdot\|)$ satisfy
$n^{(3)}(X)=0$ and $\|\cdot\|^{\ell}$ is not a polynomial for
any positive number $\ell$.\ } Indeed, for any irrational $0<a
<1$, we consider the function $\|\cdot\|_a$ defined by
$$
\|(x,y)\|_a=\left(x^2+\left(\tfrac{a}{1+a}\right)^{1+a}y^2\right)^\frac{-a}{2}
\left(x^2+\left(\tfrac{a}{1+a}\right)^{a}y^2\right)^\frac{1+a}{2} \qquad
\big((x,y)\in\R^2\setminus\{(0,0)\}\big)
$$
and $\|(0,0)\|_a=0$, which is a norm as shown in
Proposition~\ref{prop:apendix-no-simetrica} and obviously
satisfies that $\|\cdot\|_a^{\ell}$ is not a polynomial for any
positive number $\ell$. We then consider $X=(\R^2,\|\cdot\|_a)$
and the polynomial $P=(P_1,P_2)\in\Pol[3]{X}$ given by
\begin{equation*}
P(x,y)=\left(\left(\tfrac{a}{1+a}\right)^{a}\left(\tfrac{1+2a}{1+a}\right)x^2y+
\left(\tfrac{a}{1+a}\right)^{1+2a}y^3,-x^3\right) \qquad
\big((x,y)\in X\big).
\end{equation*}
Since $\|\cdot\|_a$ is differentiable on $S_X$, for $(x,y)\in
S_{X}$, the only functional $(x^*,y^*)\in S_{X^*}$ norming
$(x,y)$ is given by $\left(\frac {\partial\|\cdot\|_a}{\partial
x}(x,y)\,,\frac {\partial\|\cdot\|_a}{\partial y}(x,y)\right)$.
It is easy to check that
\begin{align*}
\frac{\partial\|\cdot\|_a}{\partial
x}(x,y)&=x^3\,A(x,y,a) \\
\frac{\partial\|\cdot\|_a}{\partial
y}(x,y)& =\left(\left(\tfrac{a}{1+a}\right)^{a}\left(\tfrac{1+2a}{1+a}\right)x^2y+
\left(\tfrac{a}{1+a}\right)^{1+2a}y^3\right)A(x,y,a)
\end{align*}
where
$$
A(x,y,a)=
\left(x^2+\left(\tfrac{a}{1+a}\right)^{1+a}y^2\right)^{\frac{-a}{2}-1}
\left(x^2+\left(\tfrac{a}{1+a}\right)^{a}y^2\right)^{\frac{1+a}{2}-1}\,.
$$
Therefore,  $x^*P_1(x,y)+y^*P_2(x,y)=0$ which implies $v(P)=0$.
\end{example}

For higher order, there are examples of absolute normalized and
symmetric norms with polynomial numerical indices equal to zero
which do not come from polynomials.

\begin{example}\label{ex:dim5}
{\slshape For every positive integer $m\geq 3$, there are
absolute normalized and symmetric norms $\|\cdot\|_{m,\theta}$
such that the spaces $X_{m,\theta}=(\R^2,
\|\cdot\|_{m,\theta})$ satisfy $n^{(2m-1)}(X_{m,\theta})=0$ and
$\|\cdot\|_{m,\theta}^{2\ell}$ is not a polynomial for any
positive number $\ell$.\ } Indeed, let $\|\cdot\|_{m,\theta}$
be defined by
$$
\|(x,y)\|_{m,\theta}=\left(x^2+y^2\right)^\frac{\theta}{2}
\left(x^{2m-2}+y^{2m-2}\right)^\frac{1-\theta}{2m-2}
\qquad \big((x,y)\in\R^2\big)
$$
where $\theta\in [0,1]$. This formula defines a norm as shown
in Proposition~\ref{prop:apendix-normas-p}. To prove that
$n^{(2m-1)}(X_{m,\theta})=0$, we define the polynomial
$P=(P_1,P_2)\in\Pol[2m-1]{X_{m,\theta}}$ by
\begin{align*}
P_1(x,y)&=\theta y \left(x^{2m-2}+
y^{2m-2}\right)+ (1-\theta)y^{2m-3}
\left(x^2+y^2\right)\\
P_2(x,y)&=- \theta x \left(x^{2m-2}+
y^{2m-2}\right) - (1-\theta)x^{2m-3}
\left(x^2+y^2\right)
\end{align*}
and we show that $v(P)=0$. Since $\|\cdot\|_{m,\theta}$ is
differentiable on $S_{X_{m,\theta}}$, for $(x,y)\in
S_{X_{m,\theta}}$ the only functional $(x^*,y^*)\in
S_{X_{m,\theta}^*}$ norming $(x,y)$ is given by $\left(\frac
{\partial\|\cdot\|_{m,\theta}}{\partial x}(x,y)\,,\frac
{\partial\|\cdot\|_{m,\theta}}{\partial y}(x,y)\right)$ and,
therefore,
\begin{align*}
x^*&=\left[ \theta x \left(x^{2m-2}+
y^{2m-2}\right) + (1-\theta)x^{2m-3}
\left(x^2+y^2\right)\right] B(x,y,m,\theta)\\
y^*&=\left[\theta y \left(x^{2m-2}+
y^{2m-2}\right)+ (1-\theta)y^{2m-3}
\left(x^2+y^2\right)\right] B(x,y,m,\theta)
\end{align*}
where
$$
B(x,y,m,\theta)=\left(x^2+y^2\right)^{\frac{\theta}{2}-1}\left(x^{2m-2}+
y^{2m-2}\right)^{\frac{1-\theta}{2m-2}-1}\,.
$$
Now, it is routine to check that $x^*P_1(x,y)+y^*P_2(x,y)=0$.
Finally, if $\theta\in [0,1]$ is chosen irrational, then
$\|\cdot\|_{m,\theta}^{2\ell}$ is not a polynomial for any
positive integer $\ell$.
\end{example}

\section{Appendix: Some norms in the plane}\label{sec:appendix}
The aim of this last section is to justify that some formulae
appearing in the past section are really norms. We start with
the norms given in Theorem~\ref{th-norm-beta} for which the
justification is direct.

\begin{prop}\label{prop:Apendix-beta}
For $\beta\in \R$, the formula
$$
\|(x,y)\|=\left(x^4+2\beta x^2y^2+y^4\right)^\frac{1}{4} \qquad \big((x,y)\in
\R^2\big)
$$
defines a norm in $\R^2$ if and only if $\beta\in [0,3]$.
\end{prop}

\begin{proof}
We start by observing that for $0\leq\beta\leq1$ we can write
$$
\|(x,y)\|=\left(\beta(x^2+y^2)^2
+(1-\beta)(x^4+y^4)\right)^{\frac{1}{4}}=
\left\|\left(\beta^{\frac14}\|(x,y)\|_2\,,(1-\beta)^{\frac14}
\|(x,y)\|_4\right)\right\|_4
$$
and so it defines a norm on $\R^2$. In case $\beta<0$, it is
easy to check that the set
$$
A=\left\{(x,y)\in \R^2 \ : \ x^4+2\beta x^2y^2+y^4\leq 1\right\}
$$
is not convex and thus $\|\cdot\|$ is not a norm. Indeed, fix
$0<\delta<(-2\beta)^{\frac12}$ and observe that the points
$$
\textstyle
\left(\frac{1}{(1+2\beta\delta^2+\delta^4)^\frac14}\,,
\frac{\delta}{(1+2\beta\delta^2+\delta^4)^\frac14}\right)\qquad\text{and}
\qquad\left(\frac{1}{(1+2\beta\delta^2+\delta^4)^\frac14}\,,
\frac{-\delta}{(1+2\beta\delta^2+\delta^4)^\frac14}\right)
$$
belong to $A$ while their midpoint
$\left(\frac{1}{(1+2\beta\delta^2+\delta^4)^\frac14}\,,
0\right)$ does not.

Finally, for $\beta \geq 1$, we consider the change of
variables given by
$$
x=\frac{u+v}{(2+2\beta)^\frac14}\qquad \text{and} \qquad
y=\frac{u-v}{(2+2\beta)^\frac14}\, ;
$$
we observe that
$$
(x^4+2\beta x^2y^2+y^4)^\frac{1}{4}=
\left(u^4+2\,\tfrac{3-\beta}{1+\beta}\,u^2v^2+v^4\right)^\frac14
$$
and that the mapping $g:[1,+\infty[\longrightarrow]-1,1]$ given
by $g(\beta)=\frac{3-\beta}{1+\beta}$ satisfies
$$
g([1,3])=[0,1] \qquad  \text{and} \qquad
g(]3,+\infty[)=]-1,0[\,.
$$
So the remaining cases $1\leq \beta\leq 3$ and $3<\beta$ are
covered respectively by the previous ones $0\leq\beta\leq 1$
and $\beta<0$.
\end{proof}

The study of the functions appearing in Examples \ref{ex:raro-dim3}
and \ref{ex:dim5} is more difficult and requires some tricky
arguments. We would like to thank Vladimir Kadets for providing us
with some crucial ideas.

We start with some folklore lemmata on convex functions. Recall
that a function $f:A\longrightarrow \R$ on a convex set $A$ is
said to be \emph{convex} if
$$
f(\lambda\,x + (1-\lambda)\,y)\leq \lambda\,f(x) + (1-\lambda)\,f(y)
\qquad \bigl(x,y\in A,\ \lambda\in[0,1]\bigr).
$$
A subset $C$ of a vector space is said to be a \emph{cone} if
$\alpha\,x + \beta\,y\in C$ for every $x,y\in C$ and every
$\alpha,\beta\in \R^+$. If $f:C\longrightarrow \R$ is positive
homogeneous, then $f$ is convex if and only if $f$ is
\emph{sublinear}, i.e.\
$$
f(x+y)\leq f(x) + f(y) \qquad (x,y\in A).
$$

\begin{lemma}\label{lemma:apendix-1}
Let $(X,\|\cdot\|)$ be a normed space, $C\subseteq X$ a cone
and let $f:C\longrightarrow \R$ be a positive homogeneous
function. If
$$
f(\lambda\,x + (1-\lambda)\,y)\leq \lambda\,f(x) + (1-\lambda)\,f(y)
\qquad \bigl(x,y\in C\cap S_X,\ \lambda\in[0,1]\bigr),
$$
then $f$ is convex on  $C$.
\end{lemma}

\begin{proof}
Since $f$ is positive homogeneous, it is enough to show that it
is sublinear. If $x,y\in C$ are non-null elements, then
$x/\|x\|$ and $y/\|y\|$ belong to $C\cap S_X$ and so
$$
\tfrac{1}{\|x\| + \|y\|}\,f(x+y)=f\left(\tfrac{\|x\|}{\|x\|+\|y\|}\,\frac{x}{\|x\|} +
\tfrac{\|y\|}{\|x\|+\|y\|}\,\frac{y}{\|y\|} \right) \leq \tfrac{1}{\|x\| + \|y\|}\,
\bigl( f(x) + f(y)\bigr).
$$
If $x=0$ or $y=0$, the result is trivial.
\end{proof}

It is well-known (see \cite[Proposition~2.2]{Tuy}, for instance) that a twice differentiable
function $f:A\longrightarrow \R$ defined on an open convex subset $A$ of $\R^d$ is convex if and
only if the Hessian matrix of $f$ is semi-definite positive. With this in mind, the following
result is completely evident.

\begin{lemma}\label{lemma:apendix-union}
Let $f:\R^d \longrightarrow \R$ be a continuous function which
is twice differentiable with the partial derivatives of second
order continuous on $\R^d\setminus\{0\}$. If there are open
convex subsets $A_1,\ldots,A_m$ such that
$\bigcup\nolimits_{i=1}^m A_i$ is dense in $\R^d$ and
$f|_{A_i}$ is convex for $i=1,\ldots,m$, then $f$ is convex on
$\R^d$.
\end{lemma}

\begin{proof}
Since $f|_{A_i}$ is convex, the Hessian matrix of $f$ is
semi-definite positive on $A_i$. Since $\bigcup\nolimits_{i=1}^m
A_i$ is dense in $\R^n$ and the partial derivatives of second
order of $f$ are continuous, we get that the Hessian matrix of $f$
is semi-definite positive on $\R^n\setminus\{0\}$. Now, for fixed
$x,y\in \R^d$ such that the segment $[x,y]$ does not contain $0$,
there is an open halfspace $S$ such that $0\notin S$ and
$[x,y]\subset S$. Since the Hessian matrix of $f$ is semi-definite
positive on $S$, we get that $f$ is convex on $S$ and so on
$[x,y]$. The remainder case in which $0\in [x,y]$ reduces to the
above one by the continuity of $f$.
\end{proof}

We finish the list of preliminary results with an obvious lemma
on convex real functions.

\begin{lemma}\label{lemma:apendix-2}
Let $I\subset \R$ be an interval, let
$\gamma,\gamma_0,\gamma_1:I\longrightarrow \R$ be twice
differentiable positive functions, and let
$\varphi=\log(\gamma)$, $\varphi_i=\log(\gamma_i)$ for $i=0,1$.
\begin{enumerate}
\item[(a)] $\gamma$ is convex if and only if $\varphi'' +
    [\varphi']^2\geq 0$. In particular, if $\varphi''\geq
    0$, then $\gamma$ is convex.
\item[(b)] If $\varphi_0''$ and $\varphi_1''$ are
    nonnegative, then for each $\theta\in [0,1]$ the
    function
$$
\gamma_\theta(t)=[\gamma_1(t)]^{\theta}[\gamma_0(t)]^{1-\theta} \qquad \bigl(t\in I\bigr)
$$
is convex.
\end{enumerate}
\end{lemma}

\begin{proof}
(a). We have clearly that
$$
 \varphi'=\frac{\gamma'}{\gamma} \qquad \text{and} \qquad
 \varphi''=\frac{\gamma''\,\gamma - [\gamma']^2}{\gamma^2},
 \qquad \text{so} \qquad \varphi'' + [\varphi']^2 = \frac{\gamma''\,\gamma}{\gamma^2}\,.
$$
Now, $\gamma$ is convex if and only if $\gamma''\geq 0$ and,
since $\gamma$ is positive, this is equivalent to $\varphi''+
[\varphi']^2\geq 0$.

(b). Writing $\varphi_\theta=\log(\gamma_\theta)$, we have that
$$
\varphi_\theta''=\theta \varphi_1'' + (1-\theta)\varphi_0''
$$
and the result follows from (a).
\end{proof}

We are now ready to state the convexity of the norms of
Examples \ref{ex:raro-dim3} and \ref{ex:dim5}.

\begin{prop}\label{prop:apendix-normas-p}
For every $p_0,p_1\in [2,+\infty[$ and every $\theta\in [0,1]$,
the function
$$
f_\theta(x,y)=\|(x,y)\|_{p_1}^\theta\,\|(x,y)\|_{p_0}^{1-\theta} \qquad \bigl(x,y\in \R\bigr)
$$
is a norm on $\R^2$.
\end{prop}

\begin{proof}
Let us define $\varphi(t)=\log(f_\theta(t,1))$ and
$\varphi_i(t)=\log \|(t,1)\|_{p_i}$ for $i=0,1$ and every
$t\in[0,1]$, and observe that
$$
\varphi_i'(t)=\frac{t^{p_i-1}}{1 + t^{p_i}} \qquad \text{and} \qquad
\varphi_i''(t)=\frac{t^{p_i-2}(p_i - 1 - t^{p_i}) }{(1+t^{p_i})^2}
\qquad \bigl(t\in[0,1],\ i=0,1\bigr).
$$
If $p_i\geq 2$, then $\varphi_i''\geq 0$ for $i=0,1$ and Lemma~\ref{lemma:apendix-2} gives us that
the function $t\longmapsto f_\theta(t,1)$ for $t\in[0,1]$ is convex. Using
Lemma~\ref{lemma:apendix-1} for $(\R^2, \|\cdot\|_{\infty})$ we have that  $f$ is convex on the
cone
$$
\{(x,y)\in \R^2\,:\, x\geq 0,\ y\geq 0,\ x\leq y\}.
$$
Since the function $f_\theta$ is absolute and symmetric, the same argument is valid in any of the
other seven cones wherein we can divide $\R^2$. Now, since $f_\theta$ is twice differentiable with
partial derivatives of second order continuous on $\R^2\setminus \{(0,0)\}$,
Lemma~\ref{lemma:apendix-union} gives us that it is convex on $\R^2$. Finally, since $f_\theta$ is
positive homogeneous and it is zero only at zero, it is a norm on $\R^2$.
\end{proof}

\begin{prop}\label{prop:apendix-no-simetrica}
For any $0<a <1$, the function $\|\cdot\|_a$ defined by
$$
\|(x,y)\|_a=\left(x^2+\left(\tfrac{a}{1+a}\right)^{1+a}y^2\right)^\frac{-a}{2}
\left(x^2+\left(\tfrac{a}{1+a}\right)^{a}y^2\right)^\frac{1+a}{2} \qquad
\big((x,y)\in\R^2\setminus\{(0,0)\}\big)
$$
and $\|(0,0)\|_a=0$, is a norm on $\R^2$.
\end{prop}

\begin{proof}
First of all, $\|\cdot\|_a$ is positive homogeneous, it is
obviously continuous on $\R^2\setminus\{(0,0)\}$ and it is also
continuous at $(0,0)$ by homogeneity. We consider the function
$$
\varphi(t)=\log\bigl(\|(t,1)\|_a\bigr) \qquad (t\in \R)
$$
and observe that
\begin{align*}
\varphi'(t)&=\frac{t^3}{\left(\left(\tfrac{a}{1+a}\right)^a
+ t^2\right)\left(\left(\tfrac{a}{1+a}\right)^{1+a} +
t^2\right)} \qquad (t\in\R )\\
\varphi''(t)&=\frac{3t^2\left(\tfrac{a}{1+a}\right)^{1+2a}+
t^4\left(\tfrac{a}{1+a}\right)^a+
t^4\left(\tfrac{a}{1+a}\right)^{1+a}-t^6}
{\left(\left(\tfrac{a}{1+a}\right)^a +
t^2\right)^2\left(\left(\tfrac{a}{1+a}\right)^{1+a} +
t^2\right)^2} \qquad (t\in\R )
\end{align*}
so we obviously obtain that
$$
\varphi''(t) +
\bigl(\varphi'(t)\bigr)^2=\frac{3t^2\left(\tfrac{a}{1+a}\right)^{1+2a}+
t^4\left(\tfrac{a}{1+a}\right)^a+
t^4\left(\tfrac{a}{1+a}\right)^{1+a}}
{\left(\left(\tfrac{a}{1+a}\right)^a +
t^2\right)^2\left(\left(\tfrac{a}{1+a}\right)^{1+a} +
t^2\right)^2} \qquad (t\in\R )
$$
Therefore, Lemma~\ref{lemma:apendix-2} gives us that the
function $t\longmapsto \|(t,1)\|_a$ for $t\in\R$ is convex and
using now Lemma~\ref{lemma:apendix-1} for $(\R^2,
|\cdot|_\eps)$ where $|(x,y)|_\eps=\max\{\eps|x|,|y|\}$, and
taking $\eps\rightarrow 0$, this implies that $\|\cdot\|_a$ is
convex on the upper halfplane. Repeating the argument by
interchanging $1$ by $-1$, we get that $\|\cdot\|_a$ is also
convex on the lower halfplane. Now,
Lemma~\ref{lemma:apendix-union} gives us that it is convex on
$\R^2$ and the homogeneity shows that $\|\cdot\|_a$ is a norm
on $\R^2$.
\end{proof}

One may wonder whether Proposition~\ref{prop:apendix-normas-p}
is true for every pair of norms on $\R^2$. The following
example shows that this is not the case even when working with
$C^\infty$ norms.

\begin{example}
{\slshape For every $\theta\in]0,1[$, there is $\eps>0$ such
that the positive homogeneous function
$$
n(x,y)=(x^2 + \eps\,y^2)^\frac{\theta}{2}\ (\eps x^2 + y^2)^\frac{1-\theta}{2}
$$
is not a norm.\ } Indeed, just observe that
$$
n(1,0)=\eps^\frac{1-\theta}{2}, \quad n(0,1)=\eps^\frac{\theta}{2} \qquad
\text{and} \qquad n(1,1)=(1+\eps)^\frac12.
$$
\end{example}

\vspace{0.5cm}

\noindent\textbf{Acknowledgement.\ } We would like to thank
Vladimir Kadets for providing us with some crucial ideas about
the content of section~\ref{sec:appendix} and Rafael Pay\'{a} for
fruitful conversations on the subject of the paper. We also
thank the anonymous referee for multiple stylistical
improvements.

\end{document}